\documentclass[10pt,english]{amsart}
\usepackage{ae,aecompl}
\usepackage[T1]{fontenc}
\usepackage{geometry}
\usepackage{babel}
\usepackage{calc}
\usepackage{amstext}
\usepackage{amsthm}
\usepackage[unicode=true,
 bookmarks=true,bookmarksnumbered=false,bookmarksopen=false,
 breaklinks=false,pdfborder={0 0 1},backref=false,colorlinks=false]
 {hyperref}
\hypersetup{pdftitle={Regularization of a backward reaction-diffusion problem}}

\makeatletter
\numberwithin{equation}{section}
\numberwithin{figure}{section}
\theoremstyle{plain}
\newtheorem{thm}{\protect\theoremname}
\theoremstyle{remark}
\newtheorem{rem}[thm]{\protect\remarkname}
\theoremstyle{plain}
\newtheorem{cor}[thm]{\protect\corollaryname}

\makeatother

\providecommand{\corollaryname}{Corollary}
\providecommand{\remarkname}{Remark}
\providecommand{\theoremname}{Theorem}

\begin{document}
\title[Spectral regularization of a time-reversed reaction-diffusion problem]{Convergence of a spectral regularization of a time-reversed reaction-diffusion problem
with high-order Sobolev-Gevrey smoothness}
\author{Vo Anh Khoa}
\address{Department of Mathematics, Florida A\&M University, Tallahassee, FL 32307, USA}
\email{anhkhoa.vo@famu.edu, vakhoa.hcmus@gmail.com}
\begin{abstract}

The present paper analyzes a spectral regularization of a time-reversed
reaction-diffusion problem with globally and locally Lipschitz nonlinearities.
This type of inverse and ill-posed problems arises in a variety
of real-world applications concerning heat conduction and tumour source
localization. In accordance with the weak solvability result for the
forward problem, we focus on the inverse problem with high-order
Sobolev-Gevrey smoothness and with Sobolev measurements.
As expected from the well-known results for the linear case, we prove
that this nonlinear spectral regularization possesses a logarithmic
rate of convergence in a high-order Sobolev norm. The proof can
be done by the verification
of variational source condition; this way
validates such a fine strategy in the framework of inverse problems
for nonlinear partial differential equations. Ultimately, we study 
a semi-discrete version of the regularization method for a class of reaction-diffusion
problems with non-degenerate nonlinearity. The convergence of this iterative
scheme is also investigated.

\end{abstract}

\maketitle

\section{Introduction}

\subsection{Statement of the inverse problem}

Here we are interested in a time-reversed reaction-diffusion model, denoted
by $\left(\mathcal{B}\right)$, with nonlinear source terms. Let $\Omega=\left[0,\ell\right]^{d}$
be a cube of $\mathbb{R}^{d}$ for $d\in\mathbb{N}$. In this context,
we consider a population density $u=u\left(x,t\right)$ where $\left(x,t\right)\in Q_{T}:=\Omega\times\left(0,T\right)$
with $T>0$, obeying the following evolution equation:
\begin{equation}
u_{t}+\mathcal{A}u=F\left(u\right)\quad\text{in }Q_{T},\label{eq:N1}
\end{equation}
associated with the periodic boundary conditions, i.e. $u\left(x+e_{i}\ell,\cdot\right)=u\left(x,\cdot\right)$
for $1\le i\le d$ where $e_{i}$ denotes the standard basis vector
for $\mathbb{R}^{d}$.

Here, $\mathcal{A}:=I-\Delta$ involves the linear second-order
differential operator and thus accounts for the anisotropic diffusion
of the population. The nonlinearity $F$ indicates either the deterministic
reaction rate or the proliferation rate for some mechanism processes.
Eventually, we complete the time-reversed model by the final condition
\begin{equation}
u\left(x,t=T\right)=g_{T}\left(x\right)\quad\text{in }\Omega.\label{eq:N2}
\end{equation}

Together with the periodic boundary condition, (\ref{eq:N1}) and
(\ref{eq:N2}) structure our time-reversed reaction-diffusion
model. In principle, the problem $\left(\mathcal{B}\right)$ is well
known to be severely ill-posed and has been investigated in a wide
range of real-world applications; see e.g. \cite{Tuan2017,Jaroudi2018,NhoHo2018}
and references cited therein for an overview of recent results and
existing models. Taking into account tumour models (see e.g. \cite{Jaroudi2018}),
the physical meaning behind this problem $\left(\mathcal{B}\right)$ is locating the tumour source
by recovering the initial density of the tumour cells. In other words,
the inverse problem we want to investigate in this paper is seeking
the initial value $u\left(x,t=0\right)=g_{0}\left(x\right)$ in a
regular tissue $\Omega$, provided that (\ref{eq:N1}) and (\ref{eq:N2})
are satisfied.

Naturally, we impose here the standard measurement on the final data
(\ref{eq:N2}), which reads as
\begin{equation}
\left\Vert g_{T}-g_{T}^{\varepsilon}\right\Vert _{L^{2}\left(\Omega\right)}\le\varepsilon,\label{eq:2.3}
\end{equation}
where $\varepsilon>0$ represents the deterministic noise level. In
practice, one attempts to get the initial function $g_{0}\left(x\right)$
from this measured data $g_{T}^{\varepsilon}$ using some
potential regularization.
\begin{rem}
\label{rem:1}As noteworthy examples for the nonlinearity $F\left(u\right)$
described in Model $\left(\mathcal{B}\right)$, we take into account
sigmoidal laws in population dynamics of cancer. The modest and simplest
tumor growth is the logistic law, which can be generalized by the
\emph{von Bertalanffy} law with $f\left(u\right)=au-bu^{N+1}$ where
$a,b$ and $N$ are specific non-negative numbers depending on every model. We also know that if the growth
rate $N$ decays exponentially, the logistic growth turns out to be
the so-called \emph{Gompertz} law $F\left(u\right)=au-bu\log u$ whenever
$u$ satisfies some additional information to avoid the singularity
of the logarithmic form. In more complex scenarios (e.g. two-species
models), one usually agrees with the \emph{de Pillis-Radunskaya} law
standing for the fractional kill rate of tumor-specific effector cells,
which reads as $F\left(u\right)=au^{N}/\left(b+u^{N}\right)$. This
is, furthermore, analogous to the generalized \emph{Michaelis\textendash Menten}
law in enzyme kinetics. Another example can also be the Frank-Kamenetskii
model in combustion theory, governed by the \emph{Arrhenius} law of
the form $F\left(u\right)=\text{exp}\left(au\right)$. Accordingly,
we see that investigating Model $\left(\mathcal{B}\right)$ can
be helpful in many areas of science and engineering.
\end{rem}

\subsection{Organization of the paper\label{subsec:Organization-of-the}}

In the following, we denote by $H_{\#}^{p}$ for $p\in\mathbb{N}$
the Hilbert spaces equipped with the standard norms and with the periodic
boundary conditions posed in the domain. For $\sigma\ge0$, we define
the Gevrey classes $G_{\sigma}^{p/2}=\mathcal{D}\left(\mathcal{A}^{p/2}e^{\sigma\mathcal{A}^{1/2}}\right)$
where the operator $\mathcal{A}$ is defined in (\ref{eq:N1}). These
are also Hilbert spaces with respect to the inner product
\[
\left\langle v,w\right\rangle _{G_{\sigma}^{p/2}}=\sum_{j\in\mathbb{Z}^{d}}v_{j}\cdot\bar{w}_{j}\left(1+\left|j\right|^{2}\right)^{p}e^{2\sigma\left(1+\left|j\right|^{2}\right)^{1/2}},
\]
and then with the corresponding norm
\[
\left\Vert v\right\Vert _{G_{\sigma}^{p/2}}=\left(\sum_{j\in\mathbb{Z}^{d}}\left|v_{j}\right|^{2}\left(1+\left|j\right|^{2}\right)^{p}e^{2\sigma\left(1+\left|j\right|^{2}\right)^{1/2}}\right)^{1/2}.
\]

With this setting, it is worth mentioning the weak solvability of
the forward problem of $\left(\mathcal{B}\right)$ where
$F$ is real analytic. Note that cf. \cite{Ferrari1998}, the real
analyticity of $F$ means that if it can be represented by $F\left(u\right)=\sum_{j=0}^{\infty}a_{j}u^{j}$
for $a_{j}\in\mathbb{R}$, then the corresponding majorising series
$\sum_{j=0}^{\infty}\left|a_{j}\right|u^{j}$ is convergent for any
$u\in\mathbb{R}$.
\begin{thm}
\label{thm:forward}\cite[Theorem 1]{Ferrari1998} Assume the initial
data $g_{0}\in H_{\#}^{p}\left(\Omega\right)$ for $p>d/2$ and the
nonlinearity $F$ is real analytic. Then there exists a time $T^{*}>0$
such that the forward model of Model $\left(\mathcal{B}\right)$
has a unique regular solution $u$ in the sense that $u\in C\left(\left[0,T^{*}\right];H_{\#}^{p}\left(\Omega\right)\right)\cap L^{2}\left(0,T^{*};\mathcal{D}\left(\mathcal{A}\right)\right)$
satisfying $u_{t}\in L^{2}\left(0,T^{*};L^{2}\left(\Omega\right)\right)$
and
\[
\left\langle u_{t},v\right\rangle _{\left(H^{1}(\Omega)\right)',H^{1}(\Omega)}+\left\langle \mathcal{A}^{1/2}u,\mathcal{A}^{1/2}v\right\rangle _{L^{2}\left(\Omega\right)}+\left\langle F\left(u\right),v\right\rangle _{L^{2}\left(\Omega\right)}=0,
\]
for all $v\in H_{\#}^{1}\left(\Omega\right)$ and for a.e. $t\in\left(0,T^{*}\right)$.
Furthermore, this regular solution satisfies $u\left(\cdot,t\right)\in G_{t}^{p/2}\left(\Omega\right)$
for $t\in\left[0,T^{*}\right)$.
\end{thm}

As a by-product of Theorem \ref{thm:forward}, we can show that $u\left(\cdot,t\right)\in G_{T^{*}}^{p/2}\left(\Omega\right)$
for all $t\ge T^{*}$ and then $F\left(u\right)\in G_{T^{*}}^{p/2}\left(\Omega\right)$
by virtue of the Taylor series expansion of $F$.  Therefore, in this work we assume our final time of observation $T$ is such that
$T< T^{*}$. Since $G_{T^{*}}^{p/2}\left(\Omega\right)\subset H^{p}\left(\Omega\right)$,
it is reasonable to assume in $\left(\mathcal{B}\right)$ that
\begin{equation}
g_{T}^{\varepsilon},g_{T}\in H^{p}\left(\Omega\right).\label{eq:regularoffinal}
\end{equation}


This paper is devoted to the convergence analysis of a modified cut-off
regularization of the inverse problem $\left(\mathcal{B}\right)$.
In this regard, we comply with the weak solvability of the forward
model to take into account the Gevrey-Sobolev source
conditions for the inverse problem. In other words, together with
Assumption (\ref{eq:regularoffinal}) we make use of the following
source condition:
\begin{equation}
u\in C\left(\left[0,T\right];H_{\#}^{p}\left(\Omega\right)\right)\cap L^{2}\left(0,T;\mathcal{D}\left(\mathcal{A}\right)\right)\;\text{and }u\left(\cdot,t\right)\in G_{t}^{p/2}\left(\Omega\right)\;\text{for }t\in\left[0,T\right),\label{eq:source}
\end{equation}

It is worth mentioning that in the forward process, we obtain the
solution in $C\left(\left[0,T^{*}\right];H_{\#}^{p}\left(\Omega\right)\right)$
with $p>d/2$, which indicates the fact that the solution belongs
to $C\left(\left[0,T\right];L^{\infty}\left(\Omega\right)\right)$. Instead of working with the real analyticity of $F$,
we can further apply the mean value theorem to get
\begin{equation}
\left|F\left(u\right)-F\left(v\right)\right|\le\left(\sup_{\left|w\right|\le M}\left|\frac{\partial F}{\partial w}\left(w\right)\right|\right)\left|u-v\right|,\label{eq:1.4-1}
\end{equation}
where for $u,v\in C\left(\left[0,T\right];L^{\infty}\left(\Omega\right)\right)$
we have denoted by
\[
M=2\max\left\{ \left\Vert u\right\Vert _{C\left(\left[0,T\right];L^{\infty}\left(\Omega\right)\right)},\left\Vert v\right\Vert _{C\left(\left[0,T\right];L^{\infty}\left(\Omega\right)\right)}\right\} >0.
\]

To this end, we will thus exploit Assumption (\ref{eq:1.4-1})
through the analysis of Section \ref{sec:Spectral-cut-off-projection}.
We also assume that there exists a continuous function $L\left(M\right)>0$
such that
\begin{equation}
\sup_{\left|w\right|\le M}\left|\frac{\partial F}{\partial w}\left(w\right)\right|\le L\left(M\right).\label{eq:1.5-1}
\end{equation}
Naturally, observing the applications mentioned in Remark \ref{rem:1}
we have
\begin{itemize}
\item for the \emph{von Bertalanffy} law: $L\left(M\right)=\left|a\right|+\left|b\right|\left(N+1\right)M^{N}$;
\item for the \emph{Gompertz} law: $L\left(M\right)=\left|a\right|+\left|b\right|M$
\item for the \emph{de Pillis-Radunskaya} law: $L\left(M\right)=\left|\frac{a}{b}\right|NM^{N-1}$
for $b\ne0$;
\item for the \emph{Arrhenius} law: $L\left(M\right)=\left|a\right|$ if
$a<0$ and $L\left(M\right)=ae^{aM}$ if $a\ge0$.
\end{itemize}
In this work, all the constants $C$ used here are independent of
the measurement error $\varepsilon$. Nonetheless, their precise values
may change from line to line and even within a single chain of estimates.
On the other side, we use either the superscript or the subscript
$\varepsilon$ to accentuate the possible dependence of the present
error on the constants.

In the first step, we derive the error bounds when $L\left(M\right)=C$
is independent of $M$. Our proof relies on the way we verify a variational
source condition. This interestingly contributes to one of their
first applications to the convergence analysis of regularization of
inverse problems for nonlinear PDEs. In this approach, we prove the
usual logarithmic-type rate of convergence of the nonlinear spectral
regularization. We remark that during the time evolving backward in
the open set $\left(0,T\right)$, the nonlinear scheme yields the
asymptotic H\"older rate. When $L\left(M\right)$ essentially depends
on $M$, the convergence is slower due to the phenomenal
growth of the quantity $L\left(M\right)$ involved in the Lipschitz
property (\ref{eq:1.4-1}). Technically, the impediment of this growth,
albeit its high-impact on the structure of the variational source
condition, can be solved using a careful choice of a cut-off function
for the nonlinearity $F$. In this sense, the regularized solution
is sought in a proper open set decided by the measurement parameter
$\varepsilon$. Eventually, a slower logarithmic-type convergence
is also expected, which extends the results in \cite{Hohage2017,Tuan2017,Nam2010}
and references cited therein. Typically, this extension illuminates
that the Gevrey regularity conventionally restricted to the convergence
analysis in \cite{Tuan2017,Nam2010} is applicable when solving
a class of time-reversed PDEs. We also add that this work partly completes
the gap of verifying variational source conditions
for a class of nonlinear PDEs, which is still questioned in \cite{Hohage2017}.

In the second theme, our concentration moves to a derivation
of an iteration-based version of the nonlinear spectral scheme with
a stronger version of (\ref{eq:1.4-1}). We end up with the convergence
of the approximate scheme by exploring the choice of the so-called
stabilization constant, which depends not only on the number of iterations,
but also on the measurement $\varepsilon$.

Prior to the setting of our approach and to closing this section,
we introduce what the variational source condition concerns in principle
as it is exploited in the skeleton of our proof.

\subsection{Background of the variational source condition\label{subsec:1.4}}

Consider the ill-posed operator equation $\mathcal{T}\left(f\right)=g$
where $\mathcal{T}$ maps from $\mathcal{D}\left(\mathcal{T}\right)\subset\mathbb{X}$
to $\mathbb{Y}$ with $\mathbb{X}$ and $\mathbb{Y}$ being Hilbert
spaces. In this Hilbert setting, we denote by $f^{\dagger}\in\mathcal{D}\left(\mathcal{T}\right)$
the exact solution and by $g^{\varepsilon}\in\mathbb{Y}$ the noisy data
satisfying $\left\Vert \mathcal{T}\left(f^{\dagger}\right)-g^{\varepsilon}\right\Vert _{\mathbb{Y}}\le\varepsilon$,
where $\varepsilon>0$ represents the deterministic noise level (cf.
e.g (\ref{eq:2.3})). There are several stable approximations to regularize
such equations. One of the most effective methods is Tikhonov regularization
in which the exact solution is approximated by a solution of the minimization
problem, viz.
\begin{equation}
f_{\alpha}^{\varepsilon}\in\underset{f\in\mathcal{D}\left(\mathcal{T}\right)}{\text{argmin}}\left[\left\Vert \mathcal{T}\left(f\right)-g^{\varepsilon}\right\Vert _{\mathbb{Y}}^{2}+\alpha\left\Vert f\right\Vert _{\mathbb{X}}^{2}\right]=:\mathcal{R}_{\alpha}\left(g^{\varepsilon}\right),\label{eq:Tikhonov}
\end{equation}
for some \emph{regularization parameter} $\alpha:=\alpha\left(\varepsilon\right)>0$.

In regularization theory, one can prove that $\mathcal{R}_{\alpha}\left(g^{\varepsilon}\right)\stackrel{\alpha\searrow0^{+}}{\longrightarrow}\mathcal{T}^{-1}$
for suitable choices of $\alpha$ in some adequate topology. Furthermore,
people usually want to specify rates of convergence, which turns out
to be the question of finding the worst case error,
since such rates are arbitrarily slow in general. Denote by $\psi:\left[0,\infty\right)\to\left[0,\infty\right)$
an \emph{index function} if it is continuous and monotonically increasing
with $\psi\left(0\right)=0$. In this sense, one aims to:

\noindent\fbox{\begin{minipage}[t]{1\columnwidth - 2\fboxsep - 2\fboxrule}%
Find (usually compact) subspaces $\mathbb{K}\subset\mathbb{X}$ such
that there exists an index function $\psi$ satisfying for all $f\in\mathbb{K}$,
it holds
\[
\sup\left\{ \left\Vert \mathcal{R}_{\alpha}\left(g^{\varepsilon}\right)-f\right\Vert _{\mathbb{X}}:\left\Vert \mathcal{T}\left(f\right)-g^{\varepsilon}\right\Vert _{\mathbb{Y}}\le\varepsilon\right\} \le\psi\left(\varepsilon\right).
\]
\end{minipage}}

Nevertheless, we often need additional \emph{a priori} assumptions
on the exact solution $f^{\dagger}$ to gain the speed of convergence,
which are called \emph{source conditions}. In the literature, the
former type of such conditions singles out that $f^{\dagger}=\psi\left(\mathcal{T}'\left[f^{\dagger}\right]*\mathcal{T}'\left[f^{\dagger}\right]\right)w$
for some $w\in\mathbb{X}$, renowned as the \emph{spectral source
condition} in the community of inverse and ill-posed problems for
several years; see e.g. \cite{Kaltenbacher2008} for a prevailing
background of source conditions. Here $\psi$ could be of the H\"older
and logarithmic types, depending on every situation. More recently,
a novel formulation for source conditions has been derived in the
sense of variational inequality, which reads as
\begin{equation}
2\left\langle f^{\dagger},f^{\dagger}-f\right\rangle _{\mathbb{X}}\le\frac{1}{2}\left\Vert f-f^{\dagger}\right\Vert _{\mathbb{X}}^{2}+\psi\left(\left\Vert \mathcal{T}\left(f\right)-\mathcal{T}\left(f^{\dagger}\right)\right\Vert _{\mathbb{Y}}^{2}\right)\quad\text{for all }f\in\mathcal{D}\left(\mathcal{T}\right)\subset\mathbb{X}.\label{eq:VSC}
\end{equation}
Cf. \cite{Hofmann2007}, it is in particular called as the \emph{variational
source condition}. Essentially, we can benefit from this source condition
not only to simplify proofs for convergence rates in Hilbert frameworks,
but also to extensively work in some Banach setting and in more general
models with distinctive measurements (cf. e.g. \cite{Hohage2014,Grasmair2010,Knig2016}).
Compared to the spectral source condition, it does not require the
Fr\'echet derivative $\mathcal{T}'$ and would thus be helpful in
certain applications where the forward operator is not sufficiently
smooth. Most importantly, variational source conditions have been
shown to be well-adapted to bounded linear operators in Hilbert spaces;
see \cite{Flemming2011}. On top of that, the variational source condition
combined with some nonlinearity condition has been formulated in e.g.
\cite{Scherzer2008} to treat nonlinear operators.

Even though there have been enormous advantages over the spectral
source conditions, variational source conditions could barely be verified
in particular models in terms of partial differential equations. So
far, with the aid of complex geometrical optics solutions the variational
source condition holds true for an inverse medium scattering problem
(cf. \cite{Hohage2015}) with $\psi$ in the logarithmic form when
the solution belongs to some Sobolev ball. In this paper, we aim to
show that it holds true for more complex and nonlinear models. It is worth mentioning that characterizations of the variational
source conditions have been obtained in \cite{Hohage2017} where the
degree of ill-posedness and the smoothness of solution are established
for arbitrary families of subspaces. Denoting by $\mathcal{L}\left(\mathbb{X}\right)$
the set of all bounded linear transformations mapping from $\mathbb{X}$
onto itself, the characterizations are provided cf. \cite[Theorem 2.1]{Hohage2017}
in the following.

\noindent\fbox{\begin{minipage}[t]{1\columnwidth - 2\fboxsep - 2\fboxrule}%
Suppose there exists a family of orthogonal projections $\mathcal{P}_{r}\in\mathcal{L}\left(\mathbb{X}\right)$
indexed by a parameter $r$ in some index set $\mathcal{J}$ such
that for some functions $\kappa,\sigma:\mathcal{J}\to\left(0,\infty\right)$
and some $c\ge0$ the following conditions hold true for all $r\in\mathcal{J}$:

$\left(\text{C}_{1}\right)$ $\left\Vert f^{\dagger}-\mathcal{P}_{r}f^{\dagger}\right\Vert _{\mathbb{X}}\le\kappa\left(r\right)$,

$\left(\text{C}_{2}\right)$ $\left\langle f^{\dagger},\mathcal{P}_{r}\left(f^{\dagger}-f\right)\right\rangle _{\mathbb{X}}\le\sigma\left(r\right)\left\Vert \mathcal{T}\left(f^{\dagger}\right)-\mathcal{T}\left(f\right)\right\Vert _{\mathbb{Y}}+c\kappa\left(r\right)\left\Vert f^{\dagger}-f\right\Vert _{\mathbb{X}}$
for any $f\in\mathcal{D}\left(\mathcal{T}\right)$ with $\left\Vert f-f^{\dagger}\right\Vert _{\mathbb{X}}\le4\left\Vert f^{\dagger}\right\Vert _{\mathbb{X}}$.

Then $f^{\dagger}$ satisfies the variational source condition (\ref{eq:VSC})
with
\[
\psi\left(\delta\right):=2\inf_{r\in\mathcal{J}}\left[\left(c+1\right)^{2}\left|\kappa\left(r\right)\right|^{2}+\sigma\left(r\right)\sqrt{\delta}\right],
\]

and $\psi$ is a concave index function if $\inf_{r\in\mathcal{J}}\kappa\left(r\right)=0$.%
\end{minipage}}

\section{Spectral cut-off projection\label{sec:Spectral-cut-off-projection}}

\subsection{Settings of the cut-off approach}

In line with the characterizations of variational source conditions,
in this section we verify and improve the cut-off projection that
has been postulated in \cite{Nam2010,Tuan2015}, using 
Assumptions (\ref{eq:2.3}), (\ref{eq:regularoffinal}), (\ref{eq:source}),
(\ref{eq:1.4-1}), (\ref{eq:1.5-1}). The nature of this approach
is that one solves the problem in a finite dimensional subspace, which
turns out to be a well-posed problem effectively controlled by the
measurement error.

We denote by $\left\{ E_{\lambda}\right\} $ the spectral family of
the positive operator $\mathcal{A}$ defined in the backward problem
$\left(\mathcal{B}\right)$ and the function $\lambda\mapsto\left\Vert E_{\lambda}v\right\Vert $
is called the \emph{spectral distribution function} of $v\in L^{2}\left(\Omega\right)$.
Thereby, with $C_{\varepsilon}>0$ being a \emph{cut-off parameter}
that will be chosen appropriately, we introduce the \emph{spectral
cut-off projection} $E_{\lambda_{\varepsilon}}:=\mathbf{1}_{\left(0,\lambda_{\varepsilon}\right]}\left(\mathcal{A}\right)$
where $\mathbf{1}_{\left(0,\lambda_{\varepsilon}\right]}$ denotes
the characteristic function of the interval $\left(0,\lambda_{\varepsilon}\right]$
with $\left\lceil \lambda_{\varepsilon}\right\rceil =C_{\varepsilon}$.
To establish the stable approximate problem of $\left(\mathcal{B}\right)$,
we adapt the cut-off projection to both the nonlinearity $F$ and
the final data $g_{T}^{\varepsilon}$. When doing so, we characterize
the abstract Gevrey class $G_{\sigma}^{p/2}$ for $\sigma\ge0$ and
$p\in\mathbb{N}$ as
\[
G_{\sigma}^{p/2}=\left\{ v\in L^{2}\left(\Omega\right):\int_{0}^{\infty}\lambda^{p}e^{2\sigma\lambda}d\left\Vert E_{\lambda}v\right\Vert ^{2}<\infty\right\} ,
\]
equipped with the norm
\[
\left\Vert v\right\Vert _{G_{\sigma}^{p/2}}^{2}=\int_{0}^{\infty}\lambda^{p}e^{2\sigma\lambda}d\left\Vert E_{\lambda}v\right\Vert ^{2}<\infty.
\]

Henceforward, for each $\varepsilon>0$ we take into account the following
approximate problem, denoted by $\left(\mathcal{B}_{\varepsilon}\right)$,
\[
\begin{cases}
u_{t}+\mathcal{A}u=\int_{0}^{\infty}\lambda_{\varepsilon}^{p/2}dE_{\lambda_{\varepsilon}}F\left(u\right) & \text{in }Q_{T},\\
u\left(T\right)=\int_{0}^{\infty}\lambda_{\varepsilon}^{p/2}dE_{\lambda_{\varepsilon}}g_{T}^{\varepsilon} & \text{in }\Omega.
\end{cases}
\]

We define a \emph{mildly weak solution} of $\left(\mathcal{B}_{\varepsilon}\right)$
(or a \emph{mild solution} for short) to be a continuous mapping $u_{\varepsilon}:\left[0,T\right]\to H^{p}\left(\Omega\right)$
obeying the integral equation
\begin{equation}
u_{\varepsilon}\left(t\right)=\int_{0}^{\infty}e^{\left(T-t\right)\lambda_{\varepsilon}}\lambda_{\varepsilon}^{p/2}dE_{\lambda_{\varepsilon}}g_{T}^{\varepsilon}-\int_{t}^{T}\int_{0}^{\infty}e^{\left(s-t\right)\lambda_{\varepsilon}}\lambda_{\varepsilon}^{p/2}dE_{\lambda_{\varepsilon}}F\left(u_{\varepsilon}\right)\left(s\right)ds.\label{eq:regu}
\end{equation}
Observe that in (\ref{eq:regu}) we actually obtain an integral equation
of the form $u_{\varepsilon}\left(t\right)=\mathcal{G}\left(u_{\varepsilon}\right)\left(t\right)$
where $\mathcal{G}$ is completely formulated by the right-hand side,
mapping from $C\left(\left[0,T\right];H^{p}\left(\Omega\right)\right)$
onto itself. Therefore, the existence and uniqueness results for (\ref{eq:regu})
can be done by the standard fixed-point argument, requiring that the
number of fixed-point iterations must be larger than the stability
magnitude usually involved in the approximate problem. Since these
results are standard, we, for simplicity, refer the reader to the
concrete reference \cite{Tuan2017} for the detailed notion of proof.

Observe again from (\ref{eq:regu}) that we can define the cut-off
operator $\mathcal{P}_{\varepsilon}\in\mathcal{L}\left(H^{p}\left(\Omega\right)\right)$
for the solution $u\left(\cdot,t_{0}\right)$ by fixing $t=t_{0}$.
In particular, it is an orthogonal projection given by
\begin{equation}
\mathcal{P}_{\varepsilon}u\left(\cdot,t_{0}\right)=\int_{0}^{\infty}e^{\left(T-t_{0}\right)\lambda_{\varepsilon}}\lambda_{\varepsilon}^{p/2}dE_{\lambda_{\varepsilon}}g_{T}^{\varepsilon}-\int_{t_{0}}^{T}\int_{0}^{\infty}e^{\left(s-t_{0}\right)\lambda_{\varepsilon}}\lambda_{\varepsilon}^{p/2}dE_{\lambda_{\varepsilon}}F\left(u\right)\left(s\right)ds.\label{eq:2.2}
\end{equation}

We also remark that in the same spirit of (\ref{eq:regu}) the nonlinear
ill-posed operator $\mathcal{T}$ (see Subsection \ref{subsec:1.4})
can be written in the following closed form:
\begin{equation}
g_{T}=\mathcal{T}g_{0}:=\int_{0}^{\infty}e^{-T\lambda}\lambda^{p/2}dE_{\lambda}g_{0}+\int_{0}^{T}\int_{0}^{\infty}e^{\left(s-T\right)\lambda}\lambda^{p/2}dE_{\lambda}F\left(u\right)\left(s\right)ds,\label{eq:operator}
\end{equation}
in which we define that $\mathcal{T}:\mathcal{D}\left(\mathcal{T}\right)\to H_{\#}^{p}\left(\Omega\right)$
with $\mathcal{D}\left(\mathcal{T}\right)=H_{\#}^{p}\left(\Omega\right)$.

\subsection{Verification of the cut-off approach: high-order Sobolev-Gevrey smoothness\label{subsec:Verification-of-the}}

From now on, we aim to verify a modified version of the spectral cut-off
regularization $\mathcal{P}_{\varepsilon}$ in a high-order smoothness
setting. As delved into the nonlinear case in \cite{Nam2010}, if
we impose the Gevrey smoothness on $u\left(\cdot,t\right)$ for all
$t\in\left[0,T\right]$, the verification of this approach is straightforward.
We thereupon obtain the convergence for $t\ge0$; compared to the
standard cut-off method which yields the convergence only for $t>0$
in the nonlinear case. This regularity assumption is very strong and
it does not seem applicable as discussed in the analysis of the forward
model (cf. Theorem \ref{thm:forward}), saying that the Gevrey smoothness
can only be available for $t>0$. It turns out that we need some modification
of this cut-off projection in this paper. On the whole, we use the
following scheme:
\begin{itemize}
\item Since $u\left(\cdot,t\right)$ satisfies the Gevrey smoothness $G_{t}^{p/2}\left(\Omega\right)$
for $t>0$, we keep solving the backward problem $\mathcal{T}u\left(t_{0}\right)=g_{T}^{\varepsilon}$
for $t_{0}>0$ by the projection $\mathcal{P}_{\varepsilon}$ constructed
from (\ref{eq:regu}).
\item Assume that $u_{t}\left(\cdot,t\right)\in H^{p}\left(\Omega\right)$. 
We find $t_{\varepsilon}\in\left(0,T\right)$ such that $u\left(t_{\varepsilon}\right)$
is an approximation of $g_{0}$ in $H^{p}\left(\Omega\right)$.
\end{itemize}
The central point of this modification is that this time we use the
Sobolev smoothness on $u_{t}$ in the neighborhood of $t=0$ to avoid
the Gevrey smoothness on $g_{0}$. This assumption is consistent
with the fact that $u_{t}\in L^{2}\left(0,T;L^{2}\left(\Omega\right)\right)$
obtained in Theorem \ref{thm:forward} and thus attainable by
the inclusion $C\left(0,T;H^{p}\left(\Omega\right)\right)\subset L^{2}\left(0,T;L^{2}\left(\Omega\right)\right)$.
Note that the Sobolev smoothness imposed on both $g_{0}$ and $u\left(\cdot,t\right)$
in $H^{p}\left(\Omega\right)$ with $p>d/2$ is essential for the
presence of the nonlinearity $F$.

\subsubsection{Case 1: $L\left(M\right)=C$ independent of $M$}
\begin{thm}
\label{thm:3}Let $p>d/2$. Assume that (\ref{eq:2.3}), (\ref{eq:regularoffinal}),
(\ref{eq:source}), (\ref{eq:1.4-1}), (\ref{eq:1.5-1}) hold. Then
for $t\in\left(0,T\right)$ the variational source condition (\ref{eq:VSC})
holds true for the operator (\ref{eq:operator}) with
\begin{equation}
\psi_{t}\left(\delta\right):=C\left(\frac{\delta}{e^{p+2\left(t-T\right)L^{2}}\left\Vert u\left(\cdot,t\right)\right\Vert _{G_{t}^{p/2}}^{2}}\right)^{\frac{t}{T+t+\frac{p}{2}}}.\label{eq:2.5}
\end{equation}
Consequently, choosing the cut-off parameter

\[
C_{\varepsilon}=\frac{C}{T+t+\frac{p}{2}}\log\frac{e^{\frac{p}{2}+\left(t-T\right)L^{2}}\left\Vert u\left(\cdot,t\right)\right\Vert _{G_{t}^{p/2}}}{\varepsilon^{1/2}},
\]
the orthogonal projection $\mathcal{P}_{\varepsilon}\in\mathcal{L}\left(H^{p}\left(\Omega\right)\right)$
defined in (\ref{eq:2.2}) is convergent with the rate
\begin{equation}
\left\Vert u_{\varepsilon}\left(\cdot,t\right)-u\left(\cdot,t\right)\right\Vert _{H^{p}\left(\Omega\right)}\le C\varepsilon^{\frac{t}{T+t+\frac{p}{2}}}.\label{eq:4.6}
\end{equation}
\end{thm}

\begin{proof}
Here we provide proof of the solution smoothness $\left(\text{C}_{1}\right)$
(cf. Subsection \ref{subsec:1.4}) whenever $u\left(\cdot,t\right)\in G_{t}^{p/2}$
for $p>d/2$ and $t\in\left(0,T\right)$. In this sense, we have
\begin{align*}
\left\Vert u\left(\cdot,t\right)-\mathcal{P}_{\varepsilon}u\left(\cdot,t\right)\right\Vert _{H^{p}\left(\Omega\right)}^{2} & =\int_{\lambda_{\varepsilon}}^{\infty}\lambda_{\varepsilon}^{p}d\left\Vert E_{\lambda}u\left(\cdot,t\right)\right\Vert ^{2}\\
 & \le e^{-2C_{\varepsilon}t}\int_{\lambda_{\varepsilon}}^{\infty}\lambda_{\varepsilon}^{p}e^{2\lambda_{\varepsilon}t}d\left\Vert E_{\lambda}u\left(\cdot,t\right)\right\Vert ^{2}\le e^{-2C_{\varepsilon}t}\left\Vert u\left(\cdot,t\right)\right\Vert _{G_{t}^{p/2}\left(\Omega\right)}^{2},
\end{align*}
which implies that $\kappa_{t}\left(C_{\varepsilon}\right)$ decays
with the rate $e^{-C_{\varepsilon}t}$. 

Concerning the degree of ill-posedness $\left(\text{C}_{2}\right)$,
we employ the equivalent formulation of the variational source condition
(\ref{eq:VSC}):
\[
\left\Vert \mathcal{P}_{r}\left(f^{\dagger}-f\right)\right\Vert _{\mathbb{X}}\le\sigma\left(r\right)\left\Vert \mathcal{T}\left(f^{\dagger}\right)-\mathcal{T}\left(f\right)\right\Vert _{\mathbb{X}}\;\text{for any }f\in\mathcal{D}\left(\mathcal{T}\right)\subset\mathbb{X}.
\]
Notice from (\ref{eq:2.3}) that the data measurement is not as smooth
as the regularity assumption we imposed. Henceforward, for any given
$\bar{u}\in C\left(\left[0,T\right];H^{p}\left(\Omega\right)\right)\backslash\left\{ 0\right\} $
it follows from (\ref{eq:regu}) that
\begin{align}
\left\Vert \mathcal{P}_{\varepsilon}\left(u\left(\cdot,t\right)-\bar{u}\left(\cdot,t\right)\right)\right\Vert _{H^{p}\left(\Omega\right)}^{2} & \le2C_{\varepsilon}^{p}\int_{0}^{\infty}e^{2\left(T-t\right)C_{\varepsilon}}d\left\Vert E_{\lambda_{\varepsilon}}\left(g_{T}^{\varepsilon}-\bar{g}_{T}^{\varepsilon}\right)\right\Vert ^{2}\label{eq:4.4}\\
 & +2\int_{t}^{T}\int_{0}^{\infty}e^{2\left(s-t\right)C_{\varepsilon}}\lambda_{\varepsilon}^{p}d\left\Vert E_{\lambda_{\varepsilon}}F\left(u-\bar{u}\right)\left(s\right)\right\Vert ^{2}ds,\nonumber 
\end{align}
where we have denoted by $\bar{g}_{T}^{\varepsilon}$ the corresponding
final data of $\bar{u}$.

In (\ref{eq:4.4}), applying the globally Lipschitz continuity of
$F$ and then multiplying the resulting estimate by the exponential
weight $e^{2tC_{\varepsilon}}$ we find that
\begin{align*}
e^{2tC_{\varepsilon}}\left\Vert \mathcal{P}_{\varepsilon}\left(u\left(\cdot,t\right)-\bar{u}\left(\cdot,t\right)\right)\right\Vert _{H^{p}\left(\Omega\right)}^{2} & \le2C_{\varepsilon}^{p}e^{2TC_{\varepsilon}}\left\Vert g_{T}^{\varepsilon}-\bar{g}_{T}^{\varepsilon}\right\Vert _{L^{2}\left(\Omega\right)}^{2}\\
 & +2L^{2}\int_{t}^{T}e^{2sC_{\varepsilon}}\left\Vert \mathcal{P}_{\varepsilon}\left(u\left(\cdot,s\right)-\bar{u}\left(\cdot,s\right)\right)\right\Vert _{H^{p}\left(\Omega\right)}^{2}ds.
\end{align*}

With the aid of the Gronwall inequality, we get
\[
\left\Vert \mathcal{P}_{\varepsilon}\left(u\left(\cdot,t\right)-\bar{u}\left(\cdot,t\right)\right)\right\Vert _{H^{p}\left(\Omega\right)}^{2}\le2C_{\varepsilon}^{p}e^{2\left(T-t\right)\left(C_{\varepsilon}+L^{2}\right)}\left\Vert g_{T}^{\varepsilon}-\bar{g}_{T}^{\varepsilon}\right\Vert _{L^{2}\left(\Omega\right)}^{2},
\]
yielding $c=0$ and that $\sigma_{t}\left(C_{\varepsilon}\right)$
increases with the rate $C_{\varepsilon}^{p}e^{\left(T-t\right)\left(C_{\varepsilon}+L^{2}\right)}$ in $\left(\text{C}_{2}\right)$.

Therefore, we conclude that the solution $u\left(\cdot,t\right)$
satisfies the variational source condition (\ref{eq:VSC}) with
\begin{equation}
\psi_{t}\left(\delta\right):=2\inf_{C_{\varepsilon}>0}\left[\left\Vert u\left(\cdot,t\right)\right\Vert _{G_{t}^{p/2}}^{2}e^{-2C_{\varepsilon}t}+C_{\varepsilon}^{\frac{p}{2}}e^{\left(T-t\right)\left(C_{\varepsilon}+L^{2}\right)}\sqrt{\delta}\right],\label{eq:2.6}
\end{equation}
then these two terms in the infimum are equal for $\delta=C_{\varepsilon}^{-p}e^{-2\left(T+t\right)C_{\varepsilon}}e^{2\left(t-T\right)L^{2}}\left\Vert u\left(\cdot,t\right)\right\Vert _{G_{t}^{p/2}}^{2}$.
With this, we take the logarithm on both sides and use the
elementary inequality that $\log\left(a\right)>1-a^{-1}$ for any
$a>0$ to rule out the choice of $C_{\varepsilon}$. Thus, we can
find some $C>0$ independent of $\delta$ such that
\[
C_{\varepsilon}=\frac{C}{T+t+\frac{p}{2}}\log\frac{e^{\frac{p}{2}+\left(t-T\right)L^{2}}\left\Vert u\left(\cdot,t\right)\right\Vert _{G_{t}^{p/2}}}{\delta^{1/2}},
\]
and plugging this into (\ref{eq:2.6}), we obtain the index function
(\ref{eq:2.5}). On top of that, cf. \cite{Grasmair2010} the error
bound (\ref{eq:4.6}) itself holds for $t\in\left(0,T\right)$.
\end{proof}
As a byproduct of Theorem \ref{thm:3} and in accordance with the
argument in \cite{Grasmair2010}, the minimizers of the Tikhonov functional
(\ref{eq:Tikhonov}) satisfy the convergence rate (\ref{eq:4.6})
if $\alpha$ can be chosen such that $-\frac{1}{2\alpha}\in\partial\left(-\psi\right)\left(4\varepsilon^{2}\right)$,
where $\partial\left(-\psi\right)$ denotes the subdifferential of
$-\psi$. By this means, we compute that $\alpha=\alpha\left(\varepsilon\right)=C\varepsilon^{\frac{2\left(T+\frac{p}{2}\right)}{T+t+\frac{p}{2}}}$,
indicating that $\alpha\searrow0^{+}$ as the measurement error tends
to 0. In addition, we obtain the following conditional stability estimate.
\begin{cor}
For $t\in\left(0,T\right)$, let $\tilde{u}\left(\cdot,t\right)$
and $\bar{u}\left(\cdot,t\right)$ be two solutions obtained from
the integral equation (\ref{eq:regu}). Under the assumptions of Theorem
\ref{thm:3}, the following estimate holds
\[
\left\Vert \tilde{u}\left(\cdot,t\right)-\bar{u}\left(\cdot,t\right)\right\Vert _{H^{p}\left(\Omega\right)}\le C\left(\frac{\left\Vert \tilde{g}_{T}-\bar{g}_{T}\right\Vert _{L^{2}\left(\Omega\right)}}{e^{p+2\left(t-T\right)L^{2}}\left\Vert u\left(\cdot,t\right)\right\Vert _{G_{t}^{p/2}}^{2}}\right)^{\frac{t}{T+t+\frac{p}{2}}},
\]
where $\tilde{g}_{T}$ and $\bar{g}_{T}$ are final conditions corresponding
to $\tilde{u}\left(\cdot,t\right)$ and $\bar{u}\left(\cdot,t\right)$,
respectively.
\end{cor}

Our modified scheme now brings into play its own feature: pointing
out an approximation candidate of our solution $g_{0}$. In fact,
it is clear to see that (\ref{eq:4.6}) is not convergent when $t=0$.
Our scheme to approximate the initial density $g_{0}$ is very simple
as we merely need to compute the ``data'' $u\left(\cdot,t_{\varepsilon}\right)$
that have been solved through the standard spectral scheme $\mathcal{P}_{\varepsilon}$.
In the following, we not only show the existence of such $t_{\varepsilon}$,
but also obtain a rigorous $\varepsilon$-dependent admissible set
that contains it, directly proving the fact that $t_{\varepsilon}\searrow0^{+}$
as $\varepsilon\searrow0^{+}$. Mathematically, one can 
establish an orthogonal projection and then mimic the way we gain
Theorem \ref{thm:3} to deduce the rate of convergence. Since, again,
this step is very trivial and somewhat self-contained by the help
of the triangle inequality, our proof below follows the conventional
way. Thus, the expression of the index function is clear through the
derivation of the convergence rate.
\begin{thm}
\label{thm:4}Let $p>d/2$. Suppose that $u_{t}\left(\cdot,t\right)\in H^{p}\left(\Omega\right)$
for $t\in\left(0,T\right)$. Let $u_{\varepsilon}\left(\cdot,t\right)$
for $t\in\left(0,T\right)$ be the unique solution obtained from the
projection $\mathcal{P}_{\varepsilon}\in\mathcal{L}\left(H^{p}\left(\Omega\right)\right)$
in (\ref{eq:2.2}). Then there always exists a sufficiently small
$\varepsilon$-dependent time $t_{\varepsilon}>0$ such that
\begin{equation}
\left\Vert u_{\varepsilon}\left(t_{\varepsilon}\right)-u\left(0\right)\right\Vert _{H^{p}\left(\Omega\right)}\le\frac{C\left(T+\frac{p}{2}\right)}{\sqrt{\left(T+\frac{p}{2}+1\right)^{2}+4\left(T+\frac{p}{2}\right)\log\left(\varepsilon^{-1}\right)}+T+\frac{p}{2}-1}.\label{eq:2.9}
\end{equation}
\end{thm}

\begin{proof}
Using (\ref{eq:4.6}) and the triangle inequality, for any $t_{\varepsilon}\in\left(0,T\right)$, we find
the following estimate :
\begin{align}
\left\Vert u_{\varepsilon}\left(t_{\varepsilon}\right)-u\left(0\right)\right\Vert _{H^{p}\left(\Omega\right)} & \le\left\Vert u_{\varepsilon}\left(t_{\varepsilon}\right)-u\left(t_{\varepsilon}\right)\right\Vert _{H^{p}\left(\Omega\right)}+\left\Vert u\left(t_{\varepsilon}\right)-u\left(0\right)\right\Vert _{H^{p}\left(\Omega\right)}\label{eq:4.7}\\
 & \le C\varepsilon^{\frac{t_{\varepsilon}}{T+t_{\varepsilon}+\frac{p}{2}}}+t_{\varepsilon}\left\Vert u_{t}\right\Vert _{H^{p}\left(\Omega\right)}.\nonumber 
\end{align}
This upper bound can be done if we can find the infimum $\frac{1}{2}\inf_{t_{\varepsilon}>0}\left(\varepsilon^{\frac{t_{\varepsilon}}{T+t_{\varepsilon}+\frac{p}{2}}}+t_{\varepsilon}\right)$
for some $t_{\varepsilon}\in\left(0,T\right)$. This indicates that
we need to solve the following algebraic problem:
\[
\varepsilon^{\frac{t_{\varepsilon}}{T+t_{\varepsilon}+\frac{p}{2}}}=t_{\varepsilon},
\]
expecting that $t_{\varepsilon}>0$ is sufficiently small. Taking
the logarithm on both sides of this equation and using the standard
inequality $\log\left(a\right)>1-a^{-1}$ for any $a>0$, we have the following inequality:
\[
t_{\varepsilon}^{2}\left(\log\varepsilon-1\right)-\left(T+\frac{p}{2}-1\right)t_{\varepsilon}+T+\frac{p}{2}>0.
\]
Due to the fact that $\left|T+\frac{p}{2}-1\right|^{2}+4\left(T+\frac{p}{2}\right)\left(1-\log\varepsilon\right)>0$
and $\log\varepsilon<1$, we deduce that
\[
t_{\varepsilon}\in\left(\frac{-b-\sqrt{b^{2}+4\left(b+1\right)\left(1-\log\varepsilon\right)}}{2\left(\log\varepsilon-1\right)},\frac{-b+\sqrt{b^{2}+4\left(b+1\right)\left(1-\log\varepsilon\right)}}{2\left(\log\varepsilon-1\right)}\right),
\]
where we have denoted by $b:=T+\frac{p}{2}-1$. Notice that taking
$\varepsilon\searrow0^{+}$, the rationalizing technique gives the following limit:
\begin{align*}
 & \lim_{\varepsilon\to0^{+}}\left(-\frac{b+\sqrt{b^{2}+4\left(b+1\right)\left(1-\log\varepsilon\right)}}{2\left(\log\varepsilon-1\right)}\right)\\
 & =\lim_{\varepsilon\to0^{+}}\frac{b^{2}-b^{2}-4\left(b+1\right)\left(1+\log\left(\varepsilon^{-1}\right)\right)}{2\left(1+\log\left(\varepsilon^{-1}\right)\right)\left(b-\sqrt{b^{2}+4\left(b+1\right)\left(1+\log\left(\varepsilon^{-1}\right)\right)}\right)}\\
 & =\lim_{\varepsilon\to0^{+}}\frac{2\left(b+1\right)}{\sqrt{b^{2}+4\left(b+1\right)\left(1+\log\left(\varepsilon^{-1}\right)\right)}-b}=0,
\end{align*}
and similarly,
\begin{align*}
\lim_{\varepsilon\to0^{+}}\left(-\frac{b-\sqrt{b^{2}+4\left(b+1\right)\log\left(\varepsilon^{-1}\right)}}{2\left(\log\varepsilon-1\right)}\right) & =\lim_{\varepsilon\to0^{+}}\frac{2\left(b+1\right)}{\sqrt{b^{2}+4\left(b+1\right)\left(1+\log\left(\varepsilon^{-1}\right)\right)}+b}=0.
\end{align*}
Hence, it follows from (\ref{eq:4.7}) that the upper bound we gain
would be of the form in (\ref{eq:2.9}).
\end{proof}

\subsubsection{Case 2: $L\left(M\right)$ dependent of $M$}

When $L\left(M\right)$ depends on $M$,
our nonlinear spectral regularization $\mathcal{P}_{\varepsilon}$
may be no longer convergent as the boundedness of the regularized
solution is not well-controlled; cf. (\ref{eq:2.5}). This means that
the quantity $M$ now has to be dependent of $\varepsilon$, saying
that $M=M_{\varepsilon}$, when getting involved in the scheme $\mathcal{P}_{\varepsilon}$.
This $\varepsilon$ dependence impacts on the structure of the index
function and further on the whole rate of convergence in Theorem \ref{thm:3}.

From now on, for some constant $\ell>0$ and for $w\in\mathbb{R}$
we introduce the cut-off function of $F$, denoted by $F_{\ell}$,
as follows:
\begin{equation}
F_{\ell}\left(w\right):=\begin{cases}
F\left(\ell\right) & \text{if }w\ge\ell,\\
F\left(w\right) & \text{if }\left|w\right|\le\ell,\\
F\left(-\ell\right) & \text{if }w\le-\ell.
\end{cases}\label{eq:2.11-1}
\end{equation}

Following this way, we modify the regularization scheme (\ref{eq:2.2})
by
\begin{equation}
\bar{\mathcal{P}}_{\varepsilon}u\left(\cdot,t_{0}\right)=\int_{0}^{\infty}e^{T\lambda_{\varepsilon}}\lambda_{\varepsilon}^{p/2}dE_{\lambda_{\varepsilon}}g_{T}^{\varepsilon}-\int_{t_{0}}^{T}\int_{0}^{\infty}e^{s\lambda_{\varepsilon}}\lambda_{\varepsilon}^{p/2}dE_{\lambda_{\varepsilon}}F_{M_{\varepsilon}}\left(u\right)\left(s\right)ds.\label{eq:2.11}
\end{equation}
and therefore, it enables us to derive the convergence rate under
a suitable choice of $M_{\varepsilon}$. Note that our cut-off function $F_{M_{\varepsilon}}$
(\ref{eq:2.11-1}) possesses the similar property in (\ref{eq:1.4-1}), i.e.
\[
\left|F_{M_{\varepsilon}}\left(u\right)-F_{M_{\varepsilon}}\left(v\right)\right|\le L\left(M_{\varepsilon}\right)\left|u-v\right|.
\]

\begin{thm}
\label{thm:5}Under the assumptions of Theorem \ref{thm:3}, we choose
$M_{\varepsilon}>0$ such that for $t\in\left(0,T\right)$
\begin{equation}
L^{2}\left(M_{\varepsilon}\right)\le\frac{1}{2\left(t-T\right)}\log\left(\varepsilon^{\beta}\right)\quad\text{for }\beta\in\left(0,\frac{1}{2}\right).\label{eq:2.13}
\end{equation}
Then the variational source condition (\ref{eq:VSC}) holds true for
the operator (\ref{eq:operator}) with
\begin{equation}
\psi_{t}\left(\delta\right):=C\left(\frac{\delta^{1-2\beta}}{e^{p}\left\Vert u\left(\cdot,t\right)\right\Vert _{G_{t}^{p/2}}^{2}}\right)^{\frac{t}{T+t+\frac{p}{2}}}.\label{eq:2.5-1}
\end{equation}
Consequently, choosing the cut-off parameter
\[
C_{\varepsilon}=\frac{C}{T+t+\frac{p}{2}}\log\frac{e^{\frac{p}{2}}\left\Vert u\left(\cdot,t\right)\right\Vert _{G_{t}^{p/2}}}{\varepsilon^{1/2-\beta}},
\]
the orthogonal projection $\mathcal{\bar{P}}_{\varepsilon}\in\mathcal{L}\left(H^{p}\left(\Omega\right)\right)$
defined in (\ref{eq:2.11}) is convergent with the rate
\begin{equation}
\left\Vert u_{\varepsilon}\left(\cdot,t\right)-u\left(\cdot,t\right)\right\Vert _{H^{p}\left(\Omega\right)}\le C\varepsilon^{\frac{t\left(1-2\beta\right)}{T+t+\frac{p}{2}}}.\label{eq:4.6-1}
\end{equation}
\end{thm}

\begin{proof}
Proof of this theorem is straightforward. In fact, starting from
the argument (\ref{eq:2.6}) we know in this case that $\delta=C_{\varepsilon}^{-p}e^{-2\left(T+t\right)C_{\varepsilon}}e^{2\left(t-T\right)L^{2}\left(M_{\varepsilon}\right)}\left\Vert u\left(\cdot,t\right)\right\Vert _{G_{t}^{p/2}}^{2}$.
Thus, the choice of $M_{\varepsilon}$ in (\ref{eq:2.13}) is apparent. This then results in the form (\ref{eq:2.5-1}) of the index function.
Hence, the proof is complete.
\end{proof}
\begin{rem}
\label{rem:Based-on-the}Based on the result obtained in Theorem \ref{thm:5},
we remark the following:
\begin{itemize}
\item As is well-known in continuous population models for single species,
we suppose the nonlinearity in the form of
\[
F\left(u\right)=u\left(1-u\right)-\frac{u^{2}}{1+u^{2}},
\]
as a prominent example of the growth-and-predation rate for the spruce
budworm which critically defoliated the balsam fir in Canada (cf.
\cite{Murray1993}). In this circumstance, it is easy to get $L\left(M\right)=1+4M>0$
and hence, we can choose that
\[
M_{\varepsilon}\le\frac{1}{4}\left(\sqrt{\frac{1}{2\left(t-T\right)}\log\left(\varepsilon^{\beta}\right)}-1\right),
\]
working with the concrete assumption $\varepsilon<e^{\frac{2\left(t-T\right)}{\beta}}$.
\item This result can be further extended to multiple-species cases.
Indeed, we consider a time evolution of the concentration of $L\in\mathbb{N}$
chemical components or constituents (molecules, radicals, ions) in
a reaction network (cf. e.g. \cite{Erdi1989}) which reads as
\[
\sum_{l=1}^{L}\alpha\left(l,r\right)X_{l}\to\sum_{l=1}^{L}\eta\left(l,r\right)X_{l}\quad\text{for }r=\overline{1,R},
\]
where $\alpha\left(l,r\right)\ge0$ and $\eta\left(l,r\right)\ge0$
are the stoichiometric coefficients or molecularities. In a constructive
manner, the mass action kinetic deterministic model of this reaction
is governed by the following system of PDEs:
\[
\partial_{t}c_{l}=\Delta c_{l}+\sum_{r=1}^{R}\left(\eta\left(l,r\right)-\alpha\left(l,r\right)\right)\prod_{l=1}^{L}c_{l}^{\alpha\left(l,r\right)}\quad\text{for }l=\overline{1,L},
\]
where $c_{l}$ is viewed as the concentration of the $l$th component
at time $t\in\left[0,T\right]$. In this regard, we obtain the vector-valued
reaction-diffusion equation (\ref{eq:N1}) in the form of $\mathbf{u}_{t}+\mathcal{A}\mathbf{u}=\mathcal{F}\left(\mathbf{u}\right)$
by denoting the vector of concentrations $\mathbf{u}\left(t\right)=\left(c_{1}\left(t\right),\ldots,c_{L}\left(t\right)\right)\in\mathbb{R}^{L}$
and
\[
\mathbf{u}^{\alpha}=\left(\mathbf{u}^{\alpha\left(\cdot,1\right)},\ldots,\mathbf{u}^{\alpha\left(\cdot,R\right)}\right)\quad\text{for }\mathbf{u}^{\alpha\left(\cdot,r\right)}=\prod_{l=1}^{L}c_{l}^{\alpha\left(l,r\right)},
\]
\[
\text{diag}\left(\mathbf{z}\right)=\begin{vmatrix}z_{1} & 0 & \cdots & \cdots & 0\\
0 & z_{2} & \ddots &  & \vdots\\
\vdots & \ddots & \ddots & \ddots & \vdots\\
\vdots &  & 0 & z_{L-1} & 0\\
0 & \cdots & \cdots & 0 & z_{L}
\end{vmatrix}\quad\text{for }\mathbf{z}\in\mathbb{R}^{L},
\]
with $\mathcal{F}\left(\mathbf{u}\right)=\left(\eta-\alpha\right)\text{diag}\left(\mathbf{u}^{\alpha}\right)$
where $\eta$ and $\alpha$ are the vectors of the stoichiometric
coefficients. By this way Theorem \ref{thm:5} can be applied.
\end{itemize}
\end{rem}

Last but not least, we state the convergence rate at $t=0$ by combining
the results of Theorems \ref{thm:4} and \ref{thm:5}.
\begin{thm}
\label{thm:7}Under the assumptions of Theorem \ref{thm:4}, let $u_{\varepsilon}\left(\cdot,t\right)$
for $t\in\left(0,T\right)$ be the unique solution obtained from the
projection $\bar{\mathcal{P}}_{\varepsilon}\in\mathcal{L}\left(H^{p}\left(\Omega\right)\right)$
in (\ref{eq:2.11}). Then there always exists a sufficiently small
$\varepsilon$-dependent time $t_{\varepsilon}>0$ such that
\begin{equation}
\left\Vert u_{\varepsilon}\left(t_{\varepsilon}\right)-u\left(0\right)\right\Vert _{H^{p}\left(\Omega\right)}\le\frac{C\left(T+\frac{p}{2}\right)}{\sqrt{\left(T+\frac{p}{2}+1\right)^{2}+4\left(T+\frac{p}{2}\right)\left(1-2\beta\right)\log\left(\varepsilon^{-1}\right)}+T+\frac{p}{2}-1}.\label{eq:2.16}
\end{equation}
\end{thm}

\begin{proof}
In the same manner as in proof of Theorem \ref{thm:4},
we prove the target estimate (\ref{eq:2.16}) by seeking the infimum $\frac{1}{2}\inf_{t_{\varepsilon}>0}\left(\varepsilon^{\frac{t_{\varepsilon}\left(1-2\beta\right)}{T+t_{\varepsilon}+\frac{p}{2}}}+t_{\varepsilon}\right)$.
Solving the algebraic problem
\[
\varepsilon^{\frac{t_{\varepsilon}\left(1-2\beta\right)}{T+t_{\varepsilon}+\frac{p}{2}}}=t_{\varepsilon},
\]
we are led to the following inequality
\[
t_{\varepsilon}^{2}\left(\left(1-2\beta\right)\log\varepsilon-1\right)-\left(T+\frac{p}{2}-1\right)t_{\varepsilon}+T+\frac{p}{2}>0.
\]

Hereby, we find the admissible interval for $t_{\varepsilon}$, as
follows:
\[
t_{\varepsilon}\in\left(\frac{-b-\sqrt{b^{2}+4\left(b+1\right)\left(1-\log\varepsilon\right)}}{2\left(\left(1-2\beta\right)\log\varepsilon-1\right)},\frac{-b+\sqrt{b^{2}+4\left(b+1\right)\left(1-\log\varepsilon\right)}}{2\left(\left(1-2\beta\right)\log\varepsilon-1\right)}\right),
\]
where $b=T+\frac{p}{2}-1$ is recalled.

Therefore, it is evident to obtain the rate (\ref{eq:2.16}). We complete
the proof of the theorem.
\end{proof}

\section{Convergence of an iterative scheme}

In this section, we reduce ourselves to the case that the nonlinearity
$F$ does not degenerate, i.e. $F\left(0\right)=0$ and there exist
positive constants $L_{0}$ and $L_{1}$ such that
\begin{equation}
0<L_{0}\le\sup_{\left|w\right|\le M}\frac{\partial F}{\partial w}\left(w\right)\le L_{1}.\label{eq:2.17}
\end{equation}
Note that we now focus on solving the regularized solution in the
open set $\left(0,T\right)$ since at $t=0$ we only need to compute
the approximation at $t=t_{\varepsilon}$. Let $1\le N\in\mathbb{N}$
and take $\omega=T/N$. In this regard, we put
\[
t_{n}=T-n\omega\quad\text{for }n=\overline{1,N}.
\]
This setting allows us to seek a numerical solution $u_{\varepsilon}^{r,n}\left(x\right)\approx u_{\varepsilon}\left(x,t_{n}\right)$
for $r\in\mathbb{N}$ in the equivalent mesh-width in $t$. The function
$u_{\varepsilon}\left(x,t_{n}\right)$ is the semi-discrete solution
of the nonlinear scheme (\ref{eq:2.11}) under scrutiny in the previous
part. Starting from the projection $\bar{\mathcal{P}}_{\varepsilon}$
with the cut-off function $F_{M_{\varepsilon}}$ used in (\ref{eq:2.11}),
the iterative scheme is designed by
\begin{equation}
\left(K+1\right)u_{\varepsilon}^{r+1,n}=Ku_{\varepsilon}^{r,n}+\int_{0}^{\infty}e^{\left(T-t_{n}\right)\lambda_{\varepsilon}}\lambda_{\varepsilon}^{p/2}dE_{\lambda_{\varepsilon}}g_{T}^{\varepsilon}-\int_{0}^{\infty}\gamma_{\varepsilon}\left(t_{n}\right)\lambda_{\varepsilon}^{p/2}dE_{\lambda_{\varepsilon}}F_{M_{\varepsilon}}\left(u_{\varepsilon}^{r,n}\right),\label{eq:3.2}
\end{equation}
with the initial guesses $u_{\varepsilon}^{r,0}\equiv g_{T}^{\varepsilon}$
and $u_{\varepsilon}^{0,n}\equiv0$. In addition, we have denoted
by $\gamma_{\varepsilon}\left(t\right):=\int_{t}^{T}e^{\left(s-t\right)\lambda_{\varepsilon}}ds\le C_{\varepsilon}^{-1}\left(e^{\left(T-t\right)C_{\varepsilon}}-1\right)=:\bar{\gamma}_{\varepsilon}\left(t\right)$
for $t\in\left(0,T\right)$.

The presence of the so-called stabilization constant $K>0$ is to
guarantee the convergence of the scheme, somehow hindered by the Lipschitz
nonlinearity $F$. In fact, we wish to designate an unconditional
numerical scheme for the projection $\bar{\mathcal{P}}_{\varepsilon}$
in the sense that the number of discretizations $N$ becomes free-to-choose
by a suitable choice of $K$. It is worth mentioning that the sequence
$\left\{ u_{\varepsilon}^{r,n}\right\} _{r\in\mathbb{N}}$ is well-defined
in the Sobolev space $H^{p}\left(\Omega\right)$ for each $\varepsilon>0$
and thus the existence and uniqueness are self-contained by virtue
of the linearity of the scheme. Note that the function $\gamma_{\varepsilon}$
is decreasing in the time argument because of $\gamma_{\varepsilon}\left(t_{n+1}\right)\ge\gamma_{\varepsilon}\left(t_{n}\right)$
for any $0\le t_{n+1}\le t_{n}\le T$, whilst it is transparently
increasing in the argument $\lambda_{\varepsilon}$.
\begin{thm}
\label{thm:9}Under the assumptions of Theorem \ref{thm:4}, let $\left\{ u_{\varepsilon}^{r,n}\right\} _{r\in\mathbb{N}}$
be the solution of the scheme (\ref{eq:3.2}). Then by choosing
\begin{equation}
K:=K\left(n,\varepsilon\right)=\max\left\{ \bar{\gamma}_{\varepsilon}\left(t_{n}\right)L_{1},e^{TC_{\varepsilon}}\right\} >0\;\text{for }n\in\mathbb{N},\label{eq:3.3}
\end{equation}
this sequence is uniformly bounded in $H^{p}\left(\Omega\right)$.
\end{thm}

\begin{proof}
In view of the decomposition
\[
u_{\varepsilon}^{r,n}=\int_{C_{\varepsilon}}^{\infty}\lambda^{p/2}dE_{\lambda}u_{\varepsilon}^{r,n}+\int_{0}^{\infty}\lambda_{\varepsilon}^{p/2}dE_{\lambda_{\varepsilon}}u_{\varepsilon}^{r,n},
\]
we can find the upper bound in $H^{p}$-norm of $u_{\varepsilon}^{r+1,n}$
as follows:
\begin{align*}
\left(K+1\right)\left\Vert u_{\varepsilon}^{r+1,n}\right\Vert _{H^{p}\left(\Omega\right)} & \le e^{\left(T-t_{n}\right)C_{\varepsilon}}\left\Vert g_{T}^{\varepsilon}\right\Vert _{H^{p}\left(\Omega\right)}+K\int_{C_{\varepsilon}}^{\infty}\lambda^{p/2}dE_{\lambda}u_{\varepsilon}^{r,n}\\
 & +\int_{0}^{\infty}\lambda_{\varepsilon}^{p/2}\left|h_{M_{\varepsilon}}\right|\left(\lambda_{\varepsilon}\right)dE_{\lambda_{\varepsilon}}u_{\varepsilon}^{r,n},
\end{align*}
where we have denoted by $h_{M_{\varepsilon}}\left[y^{j}\right]:=Ky^{j}-\gamma_{\varepsilon}\left(t_{j}\right)F_{M_{\varepsilon}}\left[y^{j}\right]$
for $y=\left(y^{j}\right)_{0\le j\le N}$.

At this stage, we remark that $h'\left[y^{j}\right]=K-\gamma_{\varepsilon}\left(t_{j}\right)F_{M_{\varepsilon}}'\left(y^{j}\right)$
and therefore, it holds $\left|h'\right|\le K-\gamma_{\varepsilon}\left(t_{j}\right)L_{0}$
for a.e. $y^{j}\in\mathbb{R}$. By the mean value theorem together
with the fact that $F\left(0\right)=0$, we estimate that
\begin{align*}
\left(K+1\right)\left\Vert u_{\varepsilon}^{r+1,n}\right\Vert _{H^{p}\left(\Omega\right)} & \le e^{\left(T-t_{n}\right)C_{\varepsilon}}\left\Vert g_{T}^{\varepsilon}\right\Vert _{H^{p}\left(\Omega\right)}+K\int_{C_{\varepsilon}}^{\infty}\lambda^{p/2}dE_{\lambda}u_{\varepsilon}^{r,n}\\
 & +\int_{0}^{\infty}\lambda_{\varepsilon}^{p/2}\left(K-\gamma_{\varepsilon}\left(t_{n}\right)L_{0}\right)dE_{\lambda_{\varepsilon}}u_{\varepsilon}^{r,n}\\
 & \le e^{\left(T-t_{n}\right)C_{\varepsilon}}\left\Vert g_{T}^{\varepsilon}\right\Vert _{H^{p}\left(\Omega\right)}+\left(K-\gamma_{\varepsilon}\left(t_{n}\right)L_{0}\right)\left\Vert u_{\varepsilon}^{r,n}\right\Vert _{H^{p}\left(\Omega\right)}\\
 & +\bar{\gamma}_{\varepsilon}\left(t_{n}\right)L_{0}\int_{C_{\varepsilon}}^{\infty}\lambda^{p/2}dE_{\lambda}u_{\varepsilon}^{r,n}.
\end{align*}

By the choice of $C_{\varepsilon}$, we gain
\[
\lim_{\varepsilon\to0^{+}}\int_{C_{\varepsilon}}^{\infty}\lambda^{p/2}dE_{\lambda}u_{\varepsilon}^{r,n}=\lim_{C_{\varepsilon}\to\infty}\int_{C_{\varepsilon}}^{\infty}\lambda^{p/2}dE_{\lambda}u_{\varepsilon}^{r,n}=0.
\]
Thus, we now enjoy the choice of $K$ in (\ref{eq:3.3}) to rule out that there exists $\mu\in\left(0,1\right)$ independent of
$r,n$ and $\varepsilon$ such that
\[
\left\Vert u_{\varepsilon}^{r+1,n}\right\Vert _{H^{p}\left(\Omega\right)}\le\mu\left(\left\Vert g_{T}^{\varepsilon}\right\Vert _{H^{p}\left(\Omega\right)}+\int_{C_{\varepsilon}}^{\infty}\lambda^{p/2}dE_{\lambda}u_{\varepsilon}^{r,n}+\left\Vert u_{\varepsilon}^{r,n}\right\Vert _{H^{p}\left(\Omega\right)}\right).
\]

By induction, we obtain
\begin{equation}
\left\Vert u_{\varepsilon}^{r,n}\right\Vert _{H^{p}\left(\Omega\right)}\le C\sum_{j=1}^{r}\mu^{j}\left\Vert g_{T}^{\varepsilon}\right\Vert _{H^{p}\left(\Omega\right)},\label{eq:3.4-1}
\end{equation}
which enables us to state that the scheme (\ref{eq:3.2}) is bounded
in $H^{p}\left(\Omega\right)$ for any $r,n$ and $\varepsilon$.
Hence, we complete the proof of the theorem.
\end{proof}
Using the Banach-Alaoglu theorem, the uniform bound deduced in the
proof of Theorem \ref{thm:9} indicates that we can extract a further
subsequence (which we relabel with the same indexes if necessary)
such that $u_{\varepsilon}^{r,n}\to u_{\varepsilon}^{n}$ weakly in
$H^{p}\left(\Omega\right)$ as $r\to\infty$. Furthermore, thanks
to the Banach-Saks theorem we know that this subsequence also admits
another subsequence such that the so-called Ces\`aro mean is strongly
convergent to $u_{\varepsilon}^{n}$ in $H^{p}\left(\Omega\right)$.
In this sense, we can write
\begin{equation}
\left\Vert \frac{1}{R}\sum_{r=0}^{R}u_{\varepsilon}^{r,n}-u_{\varepsilon}^{n}\right\Vert _{H^{p}\left(\Omega\right)}\to0\quad\text{as }R\to\infty.\label{eq:3.5}
\end{equation}

Define $w_{\varepsilon}^{R,n}:=\frac{1}{R}\sum_{r=0}^{R}u_{\varepsilon}^{r,n}$
for $R\in\mathbb{N}$. Our next step is to find the rate of convergence
of the sequence $\left\{ w_{\varepsilon}^{R,n}\right\} _{R\in\mathbb{N}}$
acquired by (\ref{eq:3.5}). By this way, we not only prove that $u_{\varepsilon}^{r,n}\to u_{\varepsilon}^{n}$
strongly in $H^{p}\left(\Omega\right)$, but also show that $u_{\varepsilon}^{n}$
is identically the semi-discrete solution of the nonlinear scheme
(\ref{eq:2.11}).
\begin{thm}
\label{thm:10}Under the assumptions of Theorem \ref{thm:4}, let
$u_{\varepsilon}\left(\cdot,t\right)$ for $t\in\left(0,T\right)$
be the unique solution obtained from the projection $\bar{\mathcal{P}}_{\varepsilon}\in\mathcal{L}\left(H^{p}\left(\Omega\right)\right)$
in (\ref{eq:2.11}). Let $\left\{ u_{\varepsilon}^{r,n}\right\} _{r\in\mathbb{N}}$
be the solution of the iterative scheme (\ref{eq:3.2}) and let $\left\{ w_{\varepsilon}^{R,n}\right\} _{R\in\mathbb{N}}$
be the Ces\`aro mean of $u_{\varepsilon}^{r,n}$ with respect to
$r$. Then for $r,n\in\mathbb{N}$ and $\varepsilon>0$ there exists
$\bar{\mu}\in\left(0,1\right)$ independent of $r,n$ and $\varepsilon$
such that the following error bound holds
\[
\left\Vert u_{\varepsilon}^{r,n}-u_{\varepsilon}\left(\cdot,t_{n}\right)\right\Vert _{H^{p}\left(\Omega\right)}\le C\bar{\mu}^{r},
\]
for $r$ sufficiently large. Moreover, for $R\in\mathbb{N}$ it holds
\[
\left\Vert w_{\varepsilon}^{R,n}-u_{\varepsilon}\left(\cdot,t_{n}\right)\right\Vert _{H^{p}\left(\Omega\right)}\le C\left(\frac{\bar{\mu}}{R}\right)^{R}.
\]
\end{thm}

\begin{proof}
Define $v_{\varepsilon}^{R+1,n}:=w_{\varepsilon}^{R+1,n}-w_{\varepsilon}^{R,n}\in H^{p}\left(\Omega\right)$.
To gain the convergence rate, we compute the difference equation:
\begin{align}
\left(R+1\right)v_{\varepsilon}^{R+1,n} & =\sum_{r=0}^{R+1}u_{\varepsilon}^{r,n}-\sum_{r=0}^{R}u_{\varepsilon}^{r,n}+\left(R+1\right)\left(\frac{1}{R+1}-\frac{1}{R}\right)\sum_{r=0}^{R}u_{\varepsilon}^{r,n}\label{eq:3.6-1}\\
 & =u_{\varepsilon}^{R+1,n}-\frac{1}{R}\sum_{r=0}^{R}u_{\varepsilon}^{r,n}=\frac{1}{R}\sum_{r=0}^{R}\left(u_{\varepsilon}^{R+1,n}-u_{\varepsilon}^{r,n}\right).\nonumber 
\end{align}

From now onward, we define $U_{\varepsilon}^{r+1,n}:=u_{\varepsilon}^{r+1,n}-u_{\varepsilon}^{r,n}\in H^{p}\left(\Omega\right)$.
Following the same way we have done in the proof of Theorem \ref{thm:9},
the function $U_{\varepsilon}^{r+1,n}$ is expressed as
\begin{align*}
\left(K+1\right)U_{\varepsilon}^{r+1,n} & =KU_{\varepsilon}^{r,n}-\int_{0}^{\infty}\gamma_{\varepsilon}\left(t_{n}\right)\lambda_{\varepsilon}^{p/2}dE_{\lambda_{\varepsilon}}F_{M_{\varepsilon}}\left(u_{\varepsilon}^{r,n}\right)\\
 & +\int_{0}^{\infty}\gamma_{\varepsilon}\left(t_{n}\right)\lambda_{\varepsilon}^{p/2}dE_{\lambda_{\varepsilon}}F_{M_{\varepsilon}}\left(u_{\varepsilon}^{r-1,n}\right).
\end{align*}

With the aid of the decomposition
\[
U_{\varepsilon}^{r,n}=\int_{C_{\varepsilon}}^{\infty}\lambda^{p/2}dE_{\lambda}U_{\varepsilon}^{r,n}+\int_{0}^{\infty}\lambda_{\varepsilon}^{p/2}dE_{\lambda_{\varepsilon}}U_{\varepsilon}^{r,n},
\]
the function $U_{\varepsilon}^{n+1}$ can be bounded from above in
the $H^{p}$-norm by
\begin{align*}
\left(K+1\right)\left\Vert U_{\varepsilon}^{r+1,n}\right\Vert _{H^{p}\left(\Omega\right)} & \le K\int_{C_{\varepsilon}}^{\infty}\lambda^{p/2}dE_{\lambda}U_{\varepsilon}^{r,n}\\
 & +\left|\int_{0}^{\infty}\lambda_{\varepsilon}^{p/2}h_{M_{\varepsilon}}\left(\lambda_{\varepsilon}\right)dE_{\lambda_{\varepsilon}}u_{\varepsilon}^{r,n}-\int_{0}^{\infty}\lambda_{\varepsilon}^{p/2}h_{M_{\varepsilon}}\left(\lambda_{\varepsilon}\right)dE_{\lambda_{\varepsilon}}u_{\varepsilon}^{r-1,n}\right|,
\end{align*}
where we have recalled that $h_{M_{\varepsilon}}\left[y^{j}\right]:=Ky^{j}-\gamma_{\varepsilon}\left(t_{j}\right)F_{M_{\varepsilon}}\left[y^{j}\right]$
for $y=\left(y^{j}\right)_{0\le j\le N}$.

Henceforward, we arrive at
\begin{align*}
\left(K+1\right)\left\Vert U_{\varepsilon}^{r+1,n}\right\Vert _{H^{p}\left(\Omega\right)} & \le\left(K-\gamma_{\varepsilon}\left(t_{n}\right)L_{0}\right)\left\Vert U_{\varepsilon}^{r,n}\right\Vert _{H^{p}\left(\Omega\right)}+\gamma_{\varepsilon}\left(t_{n}\right)L_{0}\int_{C_{\varepsilon}}^{\infty}\lambda^{p/2}dE_{\lambda}U_{\varepsilon}^{r,n}\\
 & \le K\left\Vert U_{\varepsilon}^{r,n}\right\Vert _{H^{p}\left(\Omega\right)},
\end{align*}
by virtue of the fact already known that $\left|h'\right|\le K-\gamma_{\varepsilon}\left(t_{j}\right)L_{0}$
for a.e. $y^{j}\in\mathbb{R}$ under the choice of $K$ in (\ref{eq:3.3}).
Hereby, choosing $\bar{\mu}=K\left(K+1\right)^{-1}\in\left(0,1\right)$
independent of $r,n$ and $\varepsilon$ we can conclude that
\begin{equation}
\left\Vert U_{\varepsilon}^{r+1,n}\right\Vert _{H^{p}\left(\Omega\right)}\le\bar{\mu}\left\Vert U_{\varepsilon}^{r,n}\right\Vert _{H^{p}\left(\Omega\right)},\label{eq:3.7}
\end{equation}
which, by mathematical induction, leads to
\[
\left\Vert U_{\varepsilon}^{r,n}\right\Vert _{H^{p}\left(\Omega\right)}\le C\bar{\mu}^{r}\left\Vert g_{T}^{\varepsilon}\right\Vert _{H^{p}\left(\Omega\right)}.
\]

Eventually, by the back-substitution of the function $U_{\varepsilon}^{r,n}$
this means that
\begin{align}
\left\Vert u_{\varepsilon}^{r+l,n}-u_{\varepsilon}^{r,n}\right\Vert _{H^{p}\left(\Omega\right)} & \le\left\Vert u_{\varepsilon}^{r+l,n}-u_{\varepsilon}^{r+l-1,n}\right\Vert _{H^{p}\left(\Omega\right)}+\ldots+\left\Vert u_{\varepsilon}^{r+1,n}-u_{\varepsilon}^{r,n}\right\Vert _{H^{p}\left(\Omega\right)}\label{eq:3.6}\\
 & \le\bar{\mu}^{r+l-1}\left\Vert u_{\varepsilon}^{1,n}-u_{\varepsilon}^{0,n}\right\Vert _{H^{p}\left(\Omega\right)}+\ldots+\bar{\mu}^{r}\left\Vert u_{\varepsilon}^{1,n}-u_{\varepsilon}^{0,n}\right\Vert _{H^{p}\left(\Omega\right)}\nonumber \\
 & \le\frac{\bar{\mu}^{r}\left(1-\bar{\mu}^{l}\right)}{1-\bar{\mu}}\left\Vert u_{\varepsilon}^{1,n}\right\Vert _{H^{p}\left(\Omega\right)},\nonumber 
\end{align}
which proves the fact that the sequence $\left\{ u_{\varepsilon}^{r,n}\right\} _{r\in\mathbb{N}}$
is Cauchy in $H^{p}\left(\Omega\right)$. Consequently, there exists
$\bar{u}_{\varepsilon}^{n}\in H^{p}\left(\Omega\right)$ to which $u_{\varepsilon}^{r,n}$
is strongly convergent as $r\to\infty$. In addition, it follows from
(\ref{eq:3.6}) that when $l\to\infty$,
\begin{equation}
\left\Vert u_{\varepsilon}^{r,n}-\bar{u}_{\varepsilon}^{n}\right\Vert _{H^{p}\left(\Omega\right)}\le C\bar{\mu}^{r},\label{eq:3.8}
\end{equation}
and thus, for $r$ sufficiently large we obtain the convergence of
the source term, i.e. $F\left(u_{\varepsilon}^{r,n}\right)\to F\left(\bar{u}_{\varepsilon}^{n}\right)$
strongly in $H^{p}\left(\Omega\right)$ as $r\to\infty$, whenever
$L_{1}$ is dependent of $M$ or not; see again the choice (\ref{eq:2.13})
to decide how big the iteration step $r$ needs to be.

Collectively, we have proved that $\bar{u}_{\varepsilon}^{n}$ is
identically the semi-discrete solution of the nonlinear regularization
(\ref{eq:2.11}) and further, it coincides the function $u_{\varepsilon}^{n}$
derived from the weak convergence above. Now, it suffices to close
the proof of the theorem by combining (\ref{eq:3.6-1}) and (\ref{eq:3.7}).
Essentially, we have
\[
\left\Vert v_{\varepsilon}^{R+1,n}\right\Vert _{H^{p}\left(\Omega\right)}\le\frac{1}{R\left(R+1\right)}\sum_{r=0}^{R}\left\Vert u_{\varepsilon}^{R+1,n}-u_{\varepsilon}^{r,n}\right\Vert _{H^{p}\left(\Omega\right)}\le\frac{C\bar{\mu}}{R+1}.
\]

Similar to (\ref{eq:3.8}), the Ces\`aro mean is strongly convergent
to $u_{\varepsilon}^{n}$ with the rate
\[
\left\Vert w_{\varepsilon}^{R,n}-u_{\varepsilon}^{n}\right\Vert _{H^{p}\left(\Omega\right)}\le C\left(\frac{\bar{\mu}}{R}\right)^{R}.
\]

Hence, we complete the proof of the theorem.
\end{proof}
\begin{cor}
Under the assumptions of Theorems \ref{thm:5} and \ref{thm:10},
one has the following error estimate:
\[
\left\Vert u_{\varepsilon}^{r,n}-u\left(\cdot,t_{n}\right)\right\Vert _{H^{p}\left(\Omega\right)}\le C\left(\bar{\mu}^{r}+\varepsilon^{\frac{t_{n}\left(1-2\beta\right)}{T+t_{n}+\frac{p}{2}}}\right)\quad\text{for }n=\overline{1,N-1}.
\]
\end{cor}

\begin{rem}
When $L_{1}$ is independent of $M$, the assumption (\ref{eq:2.17})
can also be found in some examples, which also aids the applicability
of the global Lipschitz case in Subsection \ref{subsec:Verification-of-the}.
Observe what have been enlisted in Subsection \ref{subsec:Organization-of-the}.
It is immediate to see that the \emph{Michaelis\textendash Menten}
law ($N=1$ in the \emph{de Pillis-Radunskaya} law) with $F\left(u\right)=au/\left(b+u\right)$ for $a,b>0$
gives
\[
\sup_{\left|w\right|\le M}\frac{\partial F}{\partial w}\left(w\right)=\sup_{\left|w\right|\le M}\frac{ab}{\left(b+w\right)^{2}}\in\left[\frac{ab}{\left(b+M\right)^{2}},\frac{a}{b}\right].
\]
\end{rem}

\section{Discussions}

We have studied a nonlinear spectral regularization to solve a semi-linear
backward parabolic equation. The scheme significantly modifies the
cut-off method developed in \cite{Nam2010,Tuan2015,Hohage2017} so
that it not only fits the nonlinear context under consideration, but
also handles certain smoothness of the solution to the forward model
in an appropriate manner. In this fashion, our proposed method is
convergent in a H\"older-type rate for $t\in\left(0,T\right)$, decreasing
backwards in time, and in a logarithmic-type rate for $t=0$. It is worth mentioning that the strong convergence
obtained in $H^{p}$ may allow us to gain the convergence of the regularized
solution on the boundary with the same rates by the standard trace
theorem.

We have also studied the convergence of an iterative method
for this nonlinear scheme. To gain the strong convergence, this approximation
works with the non-degeneracy of $F$, which is a certainly stronger
condition than those met in the analysis of the nonlinear scheme.
Essentially, we see that the property of the nonlinearity $F$ plays
a pretty much important role in deciding the convergence of the numerical
scheme, as postulated in Theorems \ref{thm:9} and \ref{thm:10}.
This also points out the most difficult issue in solving inverse problems
for nonlinear PDEs; compared to the linear cases investigated so far.
One may think that the presence of Theorem \ref{thm:9} seems unnecessary
(and so is the largeness of the stabilization constant $K$ taken
in (\ref{eq:3.3})) since it is clear that the strong convergence
of the numerical scheme has already been obtained in Theorem \ref{thm:10}.
Nevertheless, it can be understood that we have depicted a general
procedure to verify the convergence of numerical regularization schemes
in the future topics. In fact, such boundedness (as the stability
anlysis) obtained in Theorem \ref{thm:9} orientates towards the strong
convergence of the Ces\`aro mean. In some sense, this unravels the
possibility that the strong convergence of the numerical sequence
is not obtainable, generally hindered by the property of $F$. The
choice of $K$ can also be very helpful because the discretization
as well as the number of iterations become more effectively economical.
In the near future, we wish to understand deeper numerical issues
caught in particularly complex networks as presented in Remark \ref{rem:Based-on-the}.

The results of this paper can initiate the convergence analysis of
the other classes of nonlinear backward PDEs using the strategy of
verifying variational source conditions. One can also attempt to achieve
the strong convergence result in the Besov spaces for regularization
of the present backward model $\left(\mathcal{B}\right)$ in the unbounded
domain as it is in agreement with the well-posedness of the forward
problem (see e.g. \cite{Miao2004}). Observe that although the variational
source condition theory does not require the fact that the ill-posed
operator $\mathcal{T}$ admits the Fr\'echet derivative, it is self-contained
in this framework. Thus, this time we are allowed to not only derive
the variational source condition from the spectral source condition,
but also apply the iteration-based regularized Gau\ss-Newton method.
The convergence analysis of this method for the nonlinear backward
PDEs should also be considered in the forthcoming works.

\section*{Acknowledgments}

The author thanks Prof. Dr. Mohammad Kazemi (Charlotte, USA) for his support of the author's research career.

\bibliographystyle{plain}
\bibliography{mybib}

\end{document}